\documentclass[12pt]{amsart}
\usepackage[utf8]{inputenc}
\usepackage[margin=3cm]{geometry}
\usepackage{color}
\usepackage{babel}
\usepackage{amsmath}
\usepackage{amsthm}
\usepackage{amssymb}
\usepackage{graphicx}
\PassOptionsToPackage{normalem}{ulem}
\usepackage{ulem}
\usepackage[unicode=true,pdfusetitle,
 bookmarks=true,bookmarksnumbered=false,bookmarksopen=false,
 breaklinks=false,pdfborder={0 0 1},backref=page,colorlinks=true]
 {hyperref}
\hypersetup{
 linkcolor=blue, citecolor=blue}

\numberwithin{equation}{section}
\numberwithin{figure}{section}
\theoremstyle{plain}
\newtheorem{thm}{Theorem}[section]
\theoremstyle{plain}

\theoremstyle{plain}
\newtheorem{prop}[thm]{Proposition}
\theoremstyle{plain}
\newtheorem{lem}[thm]{Lemma}
\theoremstyle{remark}

\theoremstyle{definition}

\begin{document}


\newcommand\R{\mathbf{R}}%
\newcommand\N{\mathbf{N}}%
\global\long\def\prob{\mathbf{P}}%

\title{Random walk on hyperbolic groups and proper powers}
\author{Mikael de la Salle}
\date{\today}
\maketitle
\begin{abstract}
  The probability that a symmetric random walk in a hyperbolic group reaches a proper power has the same exponential rate of decay as the probability of return to the identity.
\end{abstract}

Let $\mu$ be a symmetric probability measure on a Gromov-hyperbolic group, whose support is finite, generating, and contains the identity. Let $\rho$ denote its spectral radius, that is the norm of the operator $\lambda(\mu)$ of left-convolution by $\mu$ on $\ell_2(G)$. Let $(g_n)_n$ be the random walk starting at the identity. We have the easy bound
\begin{equation}
\forall g\in\Gamma,\prob(g_{n}=g)=\langle\lambda(\mu)^{n}\delta_{e},\delta_{g}\rangle\leq\rho^{n}.\label{eq:prob_for_RW_in_terms_of_spectral_radius}
\end{equation}
By Kesten's theorem, this bound is almost optimal in the following sense:
\[\lim_n \prob(g_{n}=e)^{\frac 1 {n}}=\rho.\]

An element of $\Gamma$ is called a proper power if it is of the form $g^d$ for some $g \in \Gamma$ and $d \geq 2$. The main result of this note is that the probability that the random walk reaches a proper power at time $n$ is not exponentially larger than the probability that it reaches $0$. The necessity of the hyperbolicity assumption is discussed in \S~\ref{sec:discussion}.

\begin{thm}\label{thm:properPower} There is a polynomial $P$ such that the probability that $g_n$ is a proper power is $\leq P(n) \rho^n$.
\end{thm}
For free groups and $\mu$ the uniform measure on a free generating set, this was proved by Friedman \cite[Lemma 2.4]{MR1978881}. For general measure on free groups, we proved it with Michael Magee \cite{magee2025strongasymptoticfreenesshaar}. Michael Magee, Doron Puder and Ramon Van Handel proved it for surface groups \cite[Section 5]{magee2025strongconvergenceuniformlyrandom} using a realization of surface groups into the hyperbolic plane. Theorem~\ref{thm:properPower} is a natural generalization of this result. The idea is always the same, but the details become slightly more involved, and the proof stays in the group.

Since hyperbolic groups have the rapid decay property \cite{MR943303,MR972078}, Theorem~\ref{thm:properPower} implies that the indicator function of the set of proper powers in $\Gamma$ is tempered in the sense of \cite[Section 5]{magee2025strongasymptoticfreenesshaar} \cite[Proposition 6.3]{magee2025strongasymptoticfreenesshaar}.

The motivation for this kind of results comes from random matrix theory and more precisely strong convergence of random matrices, that is convergence of operator norms. Indeed, the novel approach to strong convergence that appeared in \cite{chen2025newapproachstrongconvergence} and that was later developped in \cite{magee2025strongasymptoticfreenesshaar,chen2024newapproachstrongconvergence,magee2025strongconvergenceuniformlyrandom}, relied on three different ingredients: (1) Taylor/Gevrey expansion of statistics of random matrix moments, (2) the polynomial method and distributions à la Schwartz, (3) proper power countings. Theorem~\ref{thm:properPower} takes care of the third step for all hyperbolic groups. But it does not yet have applications to random matrix theory beyond the already known cases.
\subsection*{Acknowledgements} I would like to thank Michael Magee, Doron Puder and Ramon Van Handel for interesting discussions, comments and suggestions about the main theorem and its proof. I also would like to thank Goulnara Arzhantseva for discussions about Golod-Shafarevich and Ol'shanski-Sapir groups.

\section{Facts on hyperbolic geometry}

A geodesic metric space $(X,d)$ is called hyperbolic if there is $\delta_0$ such that for every triple $x,y,z\in X$ and every choice of geodesics $[x,y],[y,z],[z,w]$, $[x,y]$ is contained in the $\delta_0$-neighbourhood of $[y,z] \cup [z,w]$. A finitely generated group is hyperbolic if it has a Cayley graph that is hyperbolic, and then they all are.

The Gromov product $(x,y)_z$ measures how far $z$ is from being on a geodesic between $x$ and $y$:
\[ (x,y)_z = \frac 1 2(d(x,z)+d(y,z) - d(x,y)).\]
We recall some classical consequences of hyperbolicity.
\begin{prop}\label{prop:hyperbolic} If a geodesic metric space $(X,d)$ is hyperbolic, then there is a constant $\delta$ such that:
  \begin{enumerate}
  \item\label{item:thintriangle} For every triple $x,y,z\in X$ and every choice of geodesics $[x,y],[y,z],[z,w]$, $[x,y]$ is contained in the $\delta$-neighbourhood of $[y,z] \cup [z,w]$. 
  \item\label{item:4pointcondition} $\min ( (x,z)_w,(y,z)_w) \leq (x,y)_w+\delta$ for every $x,y,z,w\in X$,
    \item\label{item:alphageod_close_to_geod} For every geodesic segment $[x,y]\subset X$ and every $z \in X$, $(x,y)_z \leq d(z,[x,y]) \leq (x,y)_z +\delta$,
  \item\label{item:path_close_to_geodesic} Every path $p$ of length $\leq n$ joining $x$ to $y$ and every $z \in X$, $d(z,p) \leq (x,y)_z + \delta(\log n+1)$,
  \end{enumerate}
\end{prop}
\begin{proof}
\eqref{item:thintriangle} is the definition; \eqref{item:4pointcondition} is \cite[Proposition III.H.1.17]{MR1744486}; \eqref{item:alphageod_close_to_geod} is \cite[Lemme 17]{MR1086650};
  \eqref{item:path_close_to_geodesic} is the combination of \cite[Proposition III.H.1.6]{MR1744486} with \eqref{item:alphageod_close_to_geod}.
\end{proof}
If $X$ is as in Proposition~\ref{prop:hyperbolic}, we have:
\begin{lem}\label{lem:criterion_for_quasi_geodesic} Let $x,y,z,w \in X$ and $\alpha \in \R_+$. 
  Assume such that $(x,z)_y \leq \alpha$, $(y,w)_z \leq \alpha$ and $d(y,z) > 2\alpha + \delta$. Then
  \[ d(x,y) + d(y,z)+d(z,w) \leq d(x,w)+4\alpha + 2\delta.\]
\end{lem}
\begin{proof} By \eqref{item:4pointcondition}, we know
  \[ \min ( (x,w)_y,(z,w)_y) \leq (x,z)_y+\delta \leq \alpha + \delta.\]
  But by the assumption $(y,w)_z \leq \alpha$ and $d(z,y)> 2\alpha+\delta$, we have
  \[ (z,w)_y \geq (z,w)_y + (y,w)_z - \alpha = d(z,y)- \alpha > \alpha +\delta,\]
  so the preceding inequality is just $(x,w)_y \leq \alpha + \delta$. We deduce
  \[ d(x,y) + d(y,z) + d(z,w) \leq d(x,y) + d(y,w) + 2\alpha \leq d(x,w) + 2\alpha + 2\delta + 2\alpha,\]
  and the lemma is proven.
\end{proof}

\section{The proof}
Let $S$ be the support of $\mu$ and $X$ the geometric realization of the Cayley graph of $\Gamma$ with respect to $S$. Write $e$ for the identity element of $\Gamma$, and let $|\cdot|=d(e,\cdot)$ denote the word-length with respect to $S$. $X$ is a hyperbolic space; pick $\delta \geq 1$ satisfying all the conclusions of Proposition~\ref{prop:hyperbolic}. I assume $\delta \geq 1$ to be able to freely bound $A+B\delta \leq (A+B)\delta$ whenever $A,B\geq 0$. All the constants and polynomials appearing in this note will depend on $\Gamma, \mu$. The proof uses the hyperbolicity in the following form. This will be of no use here, but the lemma holds with $K=O(\log n)$.
I make the choice to prove the next lemma in the vocabulary of random walks, but some might prefer the way to phrase the argument from \cite[\S 5.3]{magee2025strongconvergenceuniformlyrandom}.
\begin{lem}
\label{lem:geodesic1} Let $K,k \geq 2$. There is a polynomial $P$ and a constant $C$ such that for every $h_1,\dots,h_k \in \Gamma$ with $|h_1|+\dots+|h_k|\leq |h_1\dots h_k|+K$,
\[
\prob(g_{n}=h_1 \dots h_k)\leq P(n)\sum_{n_1+\dots+n_k = n+\lfloor C\log n\rfloor}\prod_{i=1}^k\prob(g_{n_i}  = h_i).
\]
\end{lem}
\begin{proof}
  The case of general $k$ follows from the case $k=2$ by a direct induction (that changes $P$ and $C$), so we can assume that $k=2$. The assumption is that $|h_1|+|h_2|\leq |h_1h_2|+K$ (equivalently $(e,h_1h_2)_{h_1} \leq \frac{K}{2}$), and we have to bound $\prob(g_n=h_1h_2)$. 

Let $p$ denote the path $(g_0=1,g_1,\dots,g_n)$. By \eqref{item:alphageod_close_to_geod} in Proposition~\ref{prop:hyperbolic}, if $g_n=h_1h_2$ then $d(h_1,p)\leq  \frac K 2+\delta(\log n+1)$. Write $a_0$ the integer part of $ \frac K 2+ \delta(\log n+1)$ and define $T$ as the first hitting time of the ball of radius $a_0$ around $h_1$. So by the preceding, $g_{n}=h \implies T\leq n$.
  
  Define another realization of the random walk as follows: let
  $g'_{i}$ be an independent copy of the random walk, and define
\[
\tilde{g}_{n}=\begin{cases}
g_{n} & \textrm{ if }n\leq T\\
g_{T}g'_{n-T} & \textrm{ if }T\leq n\leq T+2a_0\\
g_{T}g'_{2 a_0}g_{T}^{-1}g_{n-2a_{0}} & \textrm{ if }T+2a_{0}\leq n.
\end{cases}
\]
By the Markov property ($T$ is a stopping time) $(\tilde{g}_{n})_{n\geq0}$
is distributed as the random walk with step distribution $\mu$.

If we write $c:=\min_{s \in S} \mu(s)$, conditionally to $T$, the
event $A=\{g'_{2a_{0}}=e\textrm{ and }g'_{a_{0}}=g_{T}^{-1}h_1\}$
happens with probability $\geq c^{2a_0}$, and when it happens we have
$\tilde{g}_{n+2a_{0}}=g_{n}$ for every $n\geq T$ and
$\tilde{g}_{T+a_{0}}=h_1$. This is where we use our assumption that $\mu(e)>0$
to avoid any periodicity issues. Therefore, we can bound the
probability of the event
\[
B=\big\{\tilde{g}_{n+2a_{0}}=h_1h_2\textrm{ and }h_1\in\{\tilde{g_{k}},0\leq k\leq n+2a_{0}\}\big\}
\]
as follows 
\[
\prob(B)\geq\prob(g_{n}=h_1 h_2\textrm{ and }A)\geq c^{2a_0}\prob(g_{n}=h_1h_2).
\]
By the union bound we obtain 
\[
c^{2a_0}\prob(g_{n}=h_1h_2)\leq\sum_{k=0}^{n+2a_{0}}\prob(\tilde{g}_{n+2a_{0}}=h_1h_2\textrm{ and }\tilde{g_{k}}=h_1),
\]
which is the content of the lemma, because $2 a_0= O(\log n)$ and $c^{-2a_0}$ is bounded above by a polynomial.
\end{proof}
With Lemma~\ref{lem:geodesic1} in mind, we see that we need to understand the structure of geodesics between $e$ and a proper power. This will be described in lemma~\ref{lem:power_conjugacy_class} for powers of elements with large translation length. The following result, which will be used to deal with powers of elements with small translation length, is proved in the same way but is a bit simpler.

If $C$ is a conjugacy class in $\Gamma$, we write $|C|=\min\{|x| \mid x \in C\}$ and $C^d = \{x^d\mid x\in C\}$.
\begin{lem}\label{lem:single_conjugacy_class} Let $C$ be a conjugacy class. There is a subset $\mathcal{P}_C \subset C$ of size $O(|C|)$ such that the following holds: every $x \in C$ can be written $x=ghg^{-1}$ for $g \in \Gamma$, $h \in \mathcal{P}_C$ and $|g|+|h| + |g^{-1}| \leq |x|+ 14 \delta$.
\end{lem}
\begin{proof}
  Let $k:= |C|$. If $k=0$, $C=\{1\}$ and the lemma is clear. So we can assume $k \geq 1$. If $k \leq 9\delta$, we set $\mathcal{P}_C = \{h \in C \mid |h|\leq 9 \delta\}$. If $k>9 \delta$, we choose $h=h_1 \dots h_k \in C$ of minimal length, and we take for $\mathcal{P}_C$ the cyclic conjugates of $h$: $\mathcal{P}_C =\{ h_i \dots h_k h_1\dots h_{i-1} \mid 1 \leq i \leq k\}$. We have $|\mathcal{P}_C| \leq A k$ for $A$ the size of the ball of radius $9 \delta$.

  Among all the ways to write $x=ghg^{-1}$ with $h \in \mathcal{P}_C$, pick the one that minimizes $|g|$. We can assume $|g|>0$ as otherwise there is nothing to prove. Observe that $|h| \geq 8\delta$. This is clear if $|C|> 9 \delta$. If $|C|\leq 9 \delta$, this is true because otherwise if we write $g=g's$ with $|g'|=|g|-1$ and $s \in S$, we have $x = g'(shs^{-1}) {g'}^{-1}$ with $|shs^{-1}|\leq 9 \delta$, which contradicts the minimality of $|g|$.

  We now claim that $(e,gh)_g \leq 3 \delta$. This will follow from considering  a geodesic triangle $\{e,g,gh\}$ where, in the case $|C|\geq 9\delta$, the edge $[g,gh]$ is the $g$-translate of the geodesic from $e$ to $h$ that occur in the definition of $\mathcal{P}_C$. Assume for a contradiction that  $(e,gh)_g > 3 \delta$. Then every point on $[e,gh]$ is at distance $> 3 \delta$ from $g$, so if $x$ is the point at distance $2\delta$ from $g$ on $[e,g]$ (which exists because $d(e,g)>3\delta$), then it is at distance $>\delta$ from $[e,gh]$. By the $\delta$-hyperbolicity of the triangle $\{e,g,gh\}$, $x$ is at distance $\leq \delta$ from a point $g' \in [g,gh]$. We write $g'=gu$ where $u$ is a prefix of $h$, so that $x =g' (u^{-1} h u) {g'}^{-1}$ is an expression with $(u^{-1} h u) \in \mathcal{P}_C$ (here we use our specific choice of $[g,gh]$ in the case $|C|\geq 9\delta$). Moreover, we have
  \[ |g'| = d(e,g') \leq d(e,x) + d(x,g') = |g|-2\delta + d(x,g') \leq |g|-\delta<|g|,\]
  which contradicts the minimality of $g$. This concludes the proof that  $(e,gh)_g \leq 3 \delta$.
  
  By applying the same reasoning to $h^{-1}$, we deduce $(x,g)_{gh} = (e,gh^{-1})_g \leq 3 \delta$. Finally, we have $d(g,gh)=|h| \geq 8\delta$, so the hypotheses of Lemma~\ref{lem:criterion_for_quasi_geodesic} are satisfied for $\alpha = 3\delta$, and we get
  \[ |g|+|h|+|g^{-1}| = d(e,g)+d(g,gh)+d(gh,ghg^{-1}) \leq d(e,ghg^{-1})+14 \delta.\qedhere\]
\end{proof}
We have the following variant of Lemma~\ref{lem:single_conjugacy_class}.
\begin{lem}\label{lem:power_conjugacy_class} Let $C$ be a conjugacy class with $|C| \geq 16\delta$ and $d\geq 2$.

Every $x \in C^d$ can be written $x=gh^d g^{-1}$ and 
  \[ |g|+|h|+|h^{d-2}|+|h|+|g^{-1}| \leq |x| + 34 \delta.\]
\end{lem}
\begin{proof} Among all the ways to write $x = g h^d g^{-1}$ with $h \in C$ of minimal length, choose one that minimizes $|g|$. The difference with Lemma~\ref{lem:single_conjugacy_class} is that we do not take $h$ that minimizes $|h^d|$. We will have to pay some price for this. 


Fix a geodesic between $e$ and $h$ in $X$, and duplicate it $d$ times to obtain a path $c$ from $e$ to $h^d$. Consider also a geodesic $[e,h^d]$. The path $c$ is a $|C|$-local geodesic by minimality of $|h|$. By the assumption $|C| \geq 16\delta \geq 8\delta +1$ and \cite[Theorem III.H.1.13]{MR1744486}, we get:
  \begin{enumerate}
  \item\label{item:c_close} the image of $c$  is contained in a $2\delta$-neighbourhood of $[e,h^d]$,
  \item\label{item:geod_close} $[e,h^d]$ is contained in a $3\delta$-neighbourhood of the image of $c$,
 \item\label{item:c_quasigeodesic}
   $c$ is a $(k,2\delta)$-quasi-geodesic with $k = \frac{16\delta +  4\delta}{16\delta - 4\delta}=\frac 5 3$.
  \end{enumerate}
  Observe for further reference that the first conclusion implies that $|h|+|h^{d-1}| \leq |h^d|+4\delta$, because if $x \in [e,h^d]$ is a point at distance $\leq 2\delta$ from $h$, we have $|h^{d}| = d(e,x) + d(x,h^d) \geq d(e,h) + d(h,h^d) - 2 d(x,h) \geq |h|+|h^{d-1}|-4\delta$. By applying the same for $d-1$ (if $d>2$), we deduce
  \begin{equation}\label{eq:length_of_powers_of_h}
    |h| + |h^{d-2}|+|h| \leq |h^d|+ 8\delta.
  \end{equation}

  Our next goal is to prove that $(e,gh^d)_g \leq 6 \delta$. Assume for a contradiction that this is not the case. Then every point on $[e,gh^d]$ is at distance $> 6 \delta$ from $g$, so if $z$ is the point at distance $5\delta$ from $g$ on $[e,g]$, then it is at distance $>\delta$ from $[e,gh^d]$. By hyperbolicity $z$ is at distance $\leq \delta$ from $g[e,h^d]$, so by \eqref{item:geod_close} in Proposition~\ref{prop:hyperbolic} it is at distance $\leq 4 \delta$ from a point in $gc$: there is $t$ (which can be taken as an integer) such that $d(gc(t),z)\leq 4\delta$. Then we have $z= g' {h'}^d {g'}^{-1}$ where $g'=gc(t)$ and $h':=c(t)^{-1} h c(t)$. On the one hand, $h'$ is a cyclic conjugate of $h$ so is another representative in $C$ of minimal length. On the other hand
  \[ |g'| = d(e,gc(t)) \leq d(e,z) + d(z,gc(t)) \leq |g| - 5\delta + 4\delta <|g|,\]
  which contradicts the minimality of $g$. So we indeed have $(e,gh^d)_g \leq 6 \delta$.
  
  By applying the same reasoning to $h^{-1}$, we deduce $(x,g)_{gh^d} = (e,gh^{-d})_g \leq 6 \delta$. So we are in the situation to apply Lemma~\ref{lem:criterion_for_quasi_geodesic} with $\alpha=6\delta$ ($|h^d| \geq 2\alpha+\delta=13\delta$ is guaranteed by \eqref{eq:length_of_powers_of_h}), and we get
  \[ |g|+|h^d|+|g^{-1}| \leq |x| + 26 \delta.\]
  The lemma follows from  \eqref{eq:length_of_powers_of_h}.
\end{proof}
Before we prove Theorem~\ref{thm:properPower}, we can state a simpler result that will be used in the proof, and that combines Lemma~\ref{lem:geodesic1} and Lemma~\ref{lem:single_conjugacy_class}.
\begin{lem}\label{lem:probability_conjugacy_class} There is a polynomial $P_1$ such that for every conjugacy class $C$ and every integer $n$,
  \[\prob(g_n \in C) \leq P_1(n) \rho^n.\]
\end{lem}
\begin{proof} We use Lemma~\ref{lem:single_conjugacy_class} and its notation.
  If $|C|>n$, we have $\prob(g_n \in C)=0$ so we can assume $|C|\leq n$, so that $|\mathcal{P}_C| \leq Kn$.

  Let $x \in C$, and $g \in G$ and $h \in \mathcal{P}_C$ given by Lemma~\ref{lem:single_conjugacy_class}. By Lemma~\ref{lem:geodesic1}, there is a polynomial $Q$ and and constant $C_1$ such that, if $n'=n+\lfloor C_1 \log n\rfloor$,
  \begin{align*} \prob(g_n = x)  &\leq Q(n) \sum_{k_1+k_2+k_3 = n'} \prob(g_{k_1}=g) \prob(g_{k_2}=h) \prob(g_{k_3}=g^{-1}).
  \end{align*}
  Summing over all $x \in C$, we get
  \begin{align*} \prob(g_n \in C) &\leq \sum_{h \in \mathcal{P}_C} Q(n) \sum_{k_1+k_2+k_3 = n'}  \prob(g_{k_2}=h) \sum_{g \in \Gamma} \prob(g_{k_1}=g) \prob(g_{k_3}=g^{-1})\\
    &= \sum_{h \in \mathcal{P}_C} Q(n) \sum_{k_1+k_2+k_3 = n'} \prob(g_{k_2}=h) \prob(g_{k_1+k_3} = e)\\
    &\leq |\mathcal{P}_C| Q(n) \sum_{k_1+k_2+k_3 = n'} \rho^{k_1+k_2+k_3}.
  \end{align*}
The last line is \eqref{eq:prob_for_RW_in_terms_of_spectral_radius}. This concludes the proof, because $|\mathcal{P}_C|\leq Kn$ and the number of triples $(k_1,k_2,k_3)$ with $k_1+k_2+k_3 = n'$ is $\lesssim n^2$.
\end{proof}

\begin{proof}[Proof of Theorem~\ref{thm:properPower}] Let us write
  \[ \prob(g_n\textrm{ is a proper power}) \leq  p_1(n)+p_2(n),\]
  where $p_1(n)$ is the probability that $g_n$ is a proper power of an element in a conjugacy class with $|C|<16 \delta$ and $p_2(n)$ is the probability that $g_n$ is a proper power of an element in a conjugacy class with $|C|\geq 16\delta$. The sum might be strict because $g_n$ can be both at the same time.

  Let $C$ be a conjugacy class with $|C|<16\delta$. There are two cases. Either $\sup_d |C^d|<\infty$, in which case $\{C^d\mid d\in \N\}$ is actually finite. Otherwise, there is $K$ such that $|C^d| \geq Kd$ (\cite[Proposition III.$\Gamma$.3.15]{MR1744486} and its proof). In particular, for every $n$ the number of conjugacy classes of the form $C^d$ for a conjugacy class with $|C|<16\delta$ and containing an element of length $\leq n$ is $O(n)$. So by Lemma~\ref{lem:probability_conjugacy_class} we deduce
  \[p_1(n) = O(nP_1(n) \rho^n).\]

  If $|C| \geq 16\delta$ and $d\geq 2$, we can as in the proof of Lemma~\ref{lem:probability_conjugacy_class} apply Lemmas~\ref{lem:power_conjugacy_class} and~\ref{lem:geodesic1} to get a polynomial $Q$ and a constant $C_2$ such that, if $n'=n+\lfloor C_2 \log n\rfloor$,
\begin{multline*} \prob(g_n \in C^d) \leq \sum_{g \in \Gamma,h\in C} Q(n) \sum_{k_1+\dots + k_5 = n'} \prob(g_{k_1}=g) \\\prob(g_{k_2} = h) \prob(g_{k_3} = h^{d-2}) \prob(g_{k_4} = h) \prob(g_{k_5} = g^{-1}).
\end{multline*}

  We can bound $\sum_g \prob(g_{k_1}=g) \prob(g_{k_5}=g^{-1}) = \prob(g_{k_1+k_5}=e) \leq \rho^{k_1+k_5}$ and $\prob(g_{k_3}=h^{d-2})\leq \rho^{k_3}$ by \eqref{eq:prob_for_RW_in_terms_of_spectral_radius}. Using that $\mu$ is symmetric, we have
  \begin{align} \label{eq:musymmetric}\sum_{h \in C} \prob(g_{k_2} = h)\prob(g_{k_4} = h)& = \sum_{h \in C}\prob(g_{k_2} = h)\prob(g_{k_4} = h^{-1})\\& = \prob(g_{k_2+k_4}=e \textrm{ and } g_{k_2}\in C).\nonumber
  \end{align}
We deduce 
\[ \prob(g_n \in C^d) \leq  (n')^2Q(n) \sum_{k+\ell \leq n'} \rho^{n'-k-\ell} \prob(g_{k+\ell}=e \textrm{ and } g_{k}\in C).\]
Since we have $|C^d|\geq d$ \cite[Proposition III.$\Gamma$.3.15]{MR1744486}, we have
\[ \sum_{d \geq 2} \prob(g_n \in C^d) \leq n (n')^2Q(n) \sum_{k+\ell \leq n'} \rho^{n'-k-\ell} \prob(g_{k+\ell}=e \textrm{ and } g_{k}\in C).\]
Summing over all conjugacy classes with $|C| \geq 16\delta$, we deduce by  \eqref{eq:prob_for_RW_in_terms_of_spectral_radius}
\[ p_2(n) \leq n (n')^2Q(n) \sum_{k+\ell \leq n'} \rho^{n'-k-\ell} \prob(g_{k+\ell}=e) \leq n (n')^3 Q(n) \rho^{n'}.\]
This concludes the proof of the Theorem because $n\leq n' =O(n)$.  
\end{proof}

\section{About the hypotheses}\label{sec:discussion}

Theorem~\ref{thm:properPower} has several hypotheses: on the group $\Gamma$ and on the step probability $\mu$.
\subsection{Hypotheses on $\mu$}
The fact that $\mu$ is symmetric is really used in the proof in \eqref{eq:musymmetric} and I do not know if it can be removed.

The other assumptions (support finite, generating and containing the identity) are unessential. The proof applies without any changes if the assumption $\mu(e)>0$ is replaced by an aperiodicity assumption (there is $n$ odd such that $\mu^{\ast n}(e)>0$), and can be slightly adapted in the periodic case to obtain the result. By appoximation, $\lim_n \prob(g_n\textrm{ is a proper power})^{\frac 1 n} = \rho(\mu)$ holds for every symmetric probability measure, see \cite[Proposition 6.3]{magee2025strongasymptoticfreenesshaar}.
\subsection{Hyperbolicity hypothesis}
The assumption that $\Gamma$ is hyperbolic is not necessay for the statement to hold, but it cannot be removed either. For example, Theorem~\ref{thm:properPower} is obvious whenever $\Gamma$ is amenable ($\rho=1$), and there are plenty amenable non-hyperbolic groups. In the other direction, there are examples of groups for which Theorem~\ref{thm:properPower} fails. For example, there are non-amenable finitely generated torsion groups (groups where every element has finite order): Golod-Shafarevich groups or free Burnside groups. In such groups, every element is a proper power: $g^n =e \implies g=g^{n+1}$. The existence of finitely presented non-amenable torsion groups is a well-known open problem, but there are also finitely presented groups failing Theorem~\ref{thm:properPower}: torsion-by-amenable group \cite{MR1985031}. Indeed, if $\Gamma$ has a normal torsion subgroup $\Lambda$ such that $\Gamma/\Lambda$ is amenable, then $g_n$ is a proper power whenever $g_n$ belongs to $\Lambda$. By amenability of $\Gamma/\Lambda$, this happens with sub-exponential probability:
  \[ \lim_n \prob(g_n\textrm{ is a proper power})^{\frac 1 n}=1.\]

\bibliographystyle{amsalpha}
\bibliography{strong_convergence}

\noindent Mikael de la Salle, \\
Institut Camille Jordan, CNRS, Universit\'{e} Lyon 1, France\\
\texttt{delasalle@math.univ-lyon1.fr}\\

\end{document}